\documentclass[12pt]{amsart}

\usepackage{fullpage}
\usepackage{amsmath}
\usepackage{amssymb}

\newtheorem{theorem}{Theorem}
\newtheorem*{theorem*}{Theorem}
\newtheorem{lemma}{Lemma}
\newtheorem{proposition}{Proposition}
\newtheorem{corollary}{Corollary}
\theoremstyle{definition}
\newtheorem{definition}{Definition}
\theoremstyle{remark}
\newtheorem{remark}{Remark}
\newtheorem*{remark*}{Remark}
\newtheorem*{remarkn*}{Remark on Notation}

\newtheorem*{note*}{Note}

\begin{document}

\title{Fourier Series for Singular Measures}
\author{John E. Herr and Eric S. Weber}
\address{Department of Mathematics, Iowa State University, 396 Carver Hall, Ames, IA 50011}
\email{jherr@iastate.edu, esweber@iastate.edu}
\subjclass[2000]{Primary: 42A16, 42A38; Secondary 42C15, 28A80}
\date{\today}
\begin{abstract}
Using the Kaczmarz algorithm, we prove that for any singular Borel probability measure $\mu$ on $[0,1)$, every $f\in L^2(\mu)$ possesses a Fourier series of the form $f(x)=\sum_{n=0}^{\infty}c_ne^{2\pi inx}$.  We show that the coefficients $c_{n}$ can be computed in terms of the quantities $\widehat{f}(n) = \int_{0}^{1} f(x) e^{-2\pi i n x} d \mu(x)$.  We also demonstrate a Shannon-type sampling theorem for functions that are in a sense $\mu$-bandlimited.
\end{abstract}
\maketitle

\section{Introduction}
For a Borel probability measure $\mu$, a spectrum is a sequence $\{ \lambda_{n} \}_{n\in I}$ such that the functions $\{ e^{2 \pi i \lambda_{n} x} : n \in I \} \subset L^2(\mu)$ constitute an orthonormal basis.  If $\mu$ possesses a spectrum, we say $\mu$ is spectral, and then every $f \in L^2(\mu)$ possesses a (nonharmonic) Fourier series of the form $ f(x) = \sum_{n \in I} \langle f(x), e^{2 \pi i \lambda_{n} x} \rangle e^{2 \pi i \lambda_{n} x}$. 

In \cite{JP98}, Jorgensen and Pedersen considered the question of whether measures induced by iterated function systems on $\mathbb{R}^d$ are spectral.  Remarkably, they demonstrated that the quaternary Cantor measure $\mu_4$ is spectral.  Equally remarkably, they also showed that no three exponentials are orthogonal with respect to the ternary Cantor measure $\mu_3$, and hence $\mu_3$ is not spectral. The lack of a spectrum for $\mu_3$ motivated subsequent research to relax the orthogonality condition, instead searching for an exponential frame or Riesz basis, since an exponential frame would provide a Fourier series (see \cite{DS52}) similar to the spectral case. Though these searches have yielded partial results, it is still an open question whether $L^2(\mu_3)$ possesses an exponential frame. It is known that there exist singular measures without exponential frames. In fact,  Lai \cite{Lai12} showed that self-affine measures induced by iterated function systems with no overlap cannot possess exponential frames if the probability weights are not equal.

In this paper, we demonstrate that the Kaczmarz algorithm educes another potentially fruitful substitute for exponential spectra and exponential frames: the ``effective'' sequences  defined by Kwapie\'{n} and Mycielski \cite{KwMy01}. We show that the nonnegative integral exponentials in $L^2(\mu)$ for any singular Borel probability measure $\mu$ are such an effective sequence and that this effectivity allows us to define a Fourier series representation of any function in $L^2(\mu)$.  This recovers a result of Poltoratski\u{i} \cite{Pol93} concerning the normalized Cauchy transform.

\begin{definition}
A sequence $\{f_n\}_{n=0}^{\infty}$ in a Hilbert space $\mathbb{H}$ is said to be \textit{Bessel} if there exists a constant $B>0$ such that for any $x \in \mathbb{H}$, 
\begin{equation}\label{besselcond}
\sum_{n=0}^{\infty}\lvert\langle x,f_n\rangle\rvert^2\leq B\lVert x\rVert^2.
\end{equation}
This is equivalent to the existence of a constant $D>0$ such that
\begin{equation*}\left\lVert\sum_{n=0}^{K}c_nf_n\right\rVert\leq D\sqrt{\sum_{n=0}^{K}\lvert c_n\rvert^2}\end{equation*}
for any finite sequence $\{c_0,c_1,\ldots,c_K\}$ of complex numbers.
The sequence is called a \textit{frame} if in addition there exists a constant $A>0$ such that for any $x\in\mathbb{H}$,
\begin{equation}\label{framecond}A\lVert x\rVert^2\leq\sum_{n=0}^{\infty}\lvert\langle x,f_n\rangle\rvert^2\leq B\lVert x\rVert^2.\end{equation}
If $A=B$, then the frame is said to be \textit{tight}. If $A=B=1$, then $\{f_n\}_{n=0}^{\infty}$ is a \textit{Parseval frame}. The constant $A$ is called the \textit{lower frame bound} and the constant $B$ is called the \textit{upper frame bound} or \textit{Bessel bound}.
\end{definition}

\begin{definition}The \textit{Fourier-Stieltjes transform} of a finite Borel measure $\mu$ on $[0,1)$, denoted $\widehat{\mu}$, is defined by
\begin{equation*}\widehat{\mu}(x):=\int_{0}^{1}e^{-2\pi ixy}\,d\mu(y).\end{equation*}
\end{definition}

\subsection{Effective Sequences}
Let $\{\varphi_n\}_{n=0}^{\infty}$ be a linearly dense sequence of unit vectors in a Hilbert space $\mathbb{H}$. Given any element $x\in\mathbb{H}$, we may define a sequence $\{x_n\}_{n=0}^{\infty}$ in the following manner:
\begin{align*}x_0&=\langle x,\varphi_0\rangle \varphi_0\\
x_n&=x_{n-1}+\langle x-x_{n-1},\varphi_n\rangle \varphi_n.\end{align*}
If $\lim_{n\rightarrow\infty}\lVert x-x_n\rVert=0$ regardless of the choice of $x$, then the sequence $\{\varphi_n\}_{n=0}^{\infty}$ is said to be effective.

The above formula is known as the Kaczmarz algorithm. In 1937, Stefan Kaczmarz \cite{Kacz37} proved the effectivity of linearly dense periodic sequences in the finite-dimensional case. In 2001, these results were extended to infinite-dimensional Banach spaces under certain conditions by Kwapie\'{n} and Mycielski \cite{KwMy01}. These two also gave the following formula for the sequence $\{x_n\}_{n=0}^{\infty}$, which we state here for the Hilbert space setting: Define
\begin{align}\begin{split}\label{gs}g_0&=\varphi_0\\
g_n&=\varphi_n-\sum_{i=0}^{n-1}\langle \varphi_n,\varphi_i\rangle g_i.\end{split}\end{align}
Then
\begin{equation}\label{xnsum}x_n=\sum_{i=0}^{n}\langle x,g_i\rangle \varphi_i.\end{equation}
As shown by \cite{KwMy01}, and also more clearly for the Hilbert space setting by \cite{HalSzw05}, we have
\begin{equation*}\lVert x\rVert^2-\lim_{n\rightarrow\infty}\lVert x-x_n\rVert^2=\sum_{n=0}^{\infty}\lvert\langle x,g_n\rangle\rvert^2,\end{equation*}
from which it follows that $\{\varphi_n\}_{n=0}^{\infty}$ is effective if and only if 
\begin{equation}\label{gnframe}\sum_{n=0}^{\infty}\lvert\langle x,g_n\rangle\rvert^2=\lVert x\rVert^2.\end{equation}
That is to say, $\{\varphi_n\}_{n=0}^{\infty}$ is effective if and only if the associated sequence $\{g_n\}_{n=0}^{\infty}$ is a Parseval frame.

If $\{\varphi_n\}_{n=0}^{\infty}$ is effective, then $\eqref{xnsum}$ implies that for any $x\in \mathbb{H}$, $\sum_{i=0}^{\infty}\langle x,g_i\rangle \varphi_i$ converges to $x$ in norm, and as noted $\{g_n\}_{n=0}^{\infty}$ is a Parseval frame. This does not mean that $\{g_n\}_{n=0}^{\infty}$ and $\{\varphi_n\}_{n=0}^{\infty}$ are dual frames, since $\{\varphi_n\}_{n=0}^{\infty}$ need not even be a frame. However, $\{\varphi_n\}_{n=0}^{\infty}$ and $\{g_n\}_{n=0}^{\infty}$ are pseudo-dual in the following sense, first given by Li and Ogawa in \cite{LiOg01}:

\begin{definition}\label{pseudodef}
Let $\mathcal{H}$ be a separable Hilbert space. Two sequences $\{\varphi_n\}$ and $\{\varphi_n^\star\}$ in $\mathcal{H}$ form a pair of \textit{pseudoframes} for $\mathcal{H}$ if for all $x,y\in\mathcal{H}$, $\displaystyle\langle x,y\rangle=\sum_{n}\langle x,\varphi_n^\star\rangle\langle \varphi_n,y\rangle$.
\end{definition}
All frames are pseudoframes, but not the converse. Observe that if $x,y\in\mathbb{H}$ and $\{\varphi_n\}_{n=0}^{\infty}$ is effective, then
\begin{align*}\langle x,y\rangle&=\left\langle \sum_{m=0}^{\infty}\langle x,g_m\rangle \varphi_m,y\right\rangle\\
&=\sum_{m=0}^{\infty}\langle x,g_m\rangle\left\langle \varphi_m,y\right\rangle,
\end{align*}
and so $\{\varphi_n\}_{n=0}^{\infty}$ and $\{g_n\}_{n=0}^{\infty}$ are pseudo-dual. 

Of course, since $\{g_n\}_{n=0}^{\infty}$ is a Parseval frame, it is a true dual frame for itself.

\section{Main Results}

From this point forward, we shall use the notation $e_{\lambda}(x):=e^{2\pi i\lambda x}$. Our main result is as follows:

\begin{theorem}\label{mainthm}
If $\mu$ is a singular Borel probability measure on $[0,1)$, then the sequence $\{e_n\}_{n=0}^{\infty}$ is effective in $L^2(\mu)$. As a consequence, any element $f\in L^2(\mu)$ possesses a Fourier series $$f(x)=\sum_{n=0}^{\infty}c_n e^{2\pi inx},$$ where $$c_n=\int_{0}^{1}f(x)\overline{g_n(x)}\,d\mu(x)$$ and $\{g_n\}_{n=0}^{\infty}$ is the sequence associated to $\{e_n\}_{n=0}^{\infty}$ via Equation $\eqref{gs}$. The sum converges in norm, and Parseval's identity $\lVert f\rVert^2=\sum_{n=0}^{\infty}{\lvert c_n\rvert}^2$ holds.
\end{theorem}

Our proof proceeds in a series of lemmas. First, in order to show completeness of $\{e_n\}_{n=0}^{\infty}$, we appeal to the well-known theorem of Frigyes and Marcel Riesz \cite{Riesz16}:

\begin{theorem*}[F. and M. Riesz]
Let $\mu$ be a complex Borel measure on $[0,1)$. If $$\int_{0}^{1}e^{2\pi inx}\,d\mu(x)=0$$ for all $n\in\mathbb{N}$, then $\mu$ is absolutely continuous with respect to Lebesgue measure.
\end{theorem*}
From this theorem, we prove the desired lemma:
\begin{lemma}\label{spanlem}
If $\mu$ is a singular Borel measure on $[0,1)$, then  $\{e_n\}_{n=0}^{\infty}$ is linearly dense in $L^2(\mu)$.
\end{lemma}
\begin{proof}Assume, for the sake of contradiction, that $\overline{\text{span}}(\{e_n\}_{n=0}^{\infty})\neq L^2(\mu)$. Then there exists some $f\in L^2(\mu)$ such that $f\in\overline{\text{span}}(\{e_n\}_{n=0}^{\infty})^\perp$. Then for any $n\in\mathbb{N}$, we have
\begin{equation*}\int_{0}^{1}e^{2\pi inx}\overline{f(x)}\,d\mu(x)=0.
\end{equation*}
By the F. and M. Riesz Theorem, this implies that $\overline{f}d\mu$ is absolutely continuous with respect to Lebesgue measure $d\lambda$. Since $\overline{f}d\mu<<d\lambda$ and $\overline{f}d\mu\perp d\lambda$, it follows by uniqueness in Lebesgue's Decomposition Theorem that $\overline{f}d\mu\equiv0$. Thus, $f=0$ almost everywhere with respect to $\mu$, which is a contradiction. Therefore, $\overline{\text{span}}(\{e_n\}_{n=0}^{\infty})=L^2(\mu)$.
\end{proof}

\begin{definition}[Stationary Sequences]\label{stationdef}
A sequence $\{\varphi_k\}_{k=0}^{\infty}$ in a Hilbert space is said to be \textit{stationary} if $\langle \varphi_{k+m},\varphi_{l+m}\rangle=\langle \varphi_k,\varphi_l\rangle$ for any nonnegative integers $k$, $l$, and $m$.
\end{definition}
As noted in \cite{KwMy01}, given a stationary sequence $\{\varphi_n\}_{n=0}^{\infty}$ and $a_m$ defined by $a_m:=\langle \varphi_k,\varphi_{k+m}\rangle$, where $k$ is any nonnegative integer $k\geq-m$, Bochner's Theorem implies the existence of a unique positive measure $\sigma$ on $\mathbb{T}$ such that
\begin{equation*}a_m=\int_{\mathbb{T}}\overline{z}^m\sigma(dz)=\int_{0}^{1}e^{-2\pi imx}\,d\sigma(x)\hspace{.5cm}\text{for each }m\in\mathbb{Z}.\end{equation*}
This measure $\sigma$ is called the \textit{spectral measure} of the stationary sequence $\{\varphi_n\}$.

We shall make use of the following theorem from \cite{KwMy01}:

\begin{theorem*}[Kwapie\'{n} and Mycielski]\label{kwmythm}
A stationary sequence of unit vectors that is linearly dense in a Hilbert space is effective if and only if its spectral measure either coincides with the normalized Lebesgue measure or is singular with respect to Lebesgue measure.
\end{theorem*}

We are now ready to prove Theorem \ref{mainthm}.

\begin{proof}[Proof of Theorem \ref{mainthm}]
By Lemma \ref{spanlem}, the sequence $\{e_n\}_{n=0}^{\infty}$ is linearly dense in $L^2(\mu)$. It consists of unit vectors, because $\mu$ is a probability measure. We see that for all $k,l,m\in\mathbb{N}_0$,
\[ \langle e_{k+m},e_{l+m}\rangle = \int_{[0,1)}e^{2\pi i(k-l)x}\,d\mu(x) =\langle e_k,e_l\rangle. \]
Thus, $\{e_n\}_{n=0}^{\infty}$ is stationary in $L^2(\mu)$, and moreover, $\mu$ is its spectral measure. It then follows from the theorem of Kwapie\'{n} and Mycielski that $\{e_n\}_{n=0}^{\infty}$ is effective in $L^2(\mu)$.

Since $\{e_n\}_{n=0}^{\infty}$ is effective, given any $f\in L^2(\mu)$, we have that the Kaczmarz algorithm sequence defined recursively by
\begin{align*}f_0&=\langle f,e_0\rangle e_0\\
f_{n}&=f_{n-1}+\langle f-f_{n-1},e_n\rangle e_n\end{align*}
satisfies
\begin{equation*}\lim_{n\rightarrow\infty}\lVert f-f_n\rVert=0.\end{equation*}
We recall that
\begin{equation*}f_n=\sum_{i=0}^{n}\langle f,g_i\rangle e_i,\end{equation*}
where the sequence $\{g_n\}_{n=0}^{\infty}$ is the sequence associated to the sequence $\{e_n\}_{n=0}^{\infty}$ by $\eqref{gs}$.
Hence,
\begin{equation*}f=\sum_{i=0}^{\infty}\langle f,g_i\rangle e_i.\end{equation*}
Setting $c_n=\langle f,g_n\rangle=\int_{0}^{1}f(x)\overline{g_n(x)}\,d\mu(x)$ yields
\begin{equation}\label{maineq}f(x)=\sum_{n=0}^{\infty}c_ne^{2\pi inx},\end{equation}
where the convergence is in norm. Furthermore, since $\{e_n\}_{n=0}^{\infty}$ is effective, by $\eqref{gnframe}$ $\{g_n\}_{n=0}^{\infty}$ is a Parseval frame. Thus,
$$\sum_{n=0}^{\infty}\lvert c_n\rvert^2=\sum_{n=0}^{\infty}\lvert\langle f,g_n\rangle\rvert^2=\lVert f\rVert^2.$$
 This completes the proof.
\end{proof}

Since the ternary Cantor measure $\mu_3$ is a singular probability measure, Theorem $\ref{mainthm}$ demonstrates that any $f\in L^2(\mu_3)$ possesses a Fourier series of the form prescribed by the theorem. This comes despite the fact that $\mu_3$ does not possess an orthogonal basis of exponentials. It is still unknown whether $L^2(\mu_3)$ even possesses an exponential frame. 

The sequence $\{e_n\}_{n=0}^{\infty}$ of exponentials is effective in $L^2(\mu)$ for all singular Borel probability measures $\mu$, but it is Bessel in none of them. Indeed, if it were Bessel, $\mu$ would be absolutely continuous rather than singular by Theorem 3.10 of \cite{DHW14}.  Therefore, it is not possible for $\{e_{n}\}_{n=0}^{\infty}$ to be a frame in $L^2(\mu)$.  However, by a remark in \cite{LiOg01}, since $\{ e_{n} \}_{n=0}^{\infty}$ is pseudo-dual to the (in this case Parseval) frame $\{g_{n}\}_{n=0}^{\infty}$, the upper frame bound for $\{g_n\}_{n=0}^{\infty}$ implies a lower frame bound for $\{e_n\}_{n=0}^{\infty}$.

Moreover, some of the examples in \cite{Lai12} of measures that do not possess an exponential frame are singular, and hence if we normalize them to be probability measures, Theorem \ref{mainthm} applies.

We shall give a somewhat more explicit formula for the coefficients $c_n$. We will require a lemma to do this, but first we discuss some notation:

\begin{remarkn*}Recall that a composition of a positive integer $n$ is an ordered arrangement of positive integers that sum to $n$. Whereas for a partition the order in which the terms appear does not matter, two sequences having the same terms but in a different order constitute different compositions. We will think of compositions of $n$ as tuples of positive integers whose entries sum to $n$. The set of compositions of $n$ will be denoted $P_n$. In other words,
$$P_n:=\left\{(p_1,p_2,\ldots,p_k)\mid k\in\mathbb{N},(p_1,p_2,\ldots,p_k)\in\mathbb{N}^k, p_1+p_2+\cdots+p_k=n\right\}.$$
Thus, we have $P_1=\{(1)\}$, $P_2=\{(2),(1,1)\}$, $P_3=\{(3),(1,2),(2,1),(1,1,1)\}$, etc. The length of a tuple $p\in P_n$ will be denoted $l(p)$, i.e. $p=(p_1,p_2,\ldots,p_{l(p)})\in\mathbb{N}^{l(p)}$.

\end{remarkn*}

\begin{lemma}\label{gnformula} Let $\mu$ be a Borel probability measure on $[0,1)$ with Fourier-Stieltjes transform $\widehat{\mu}$. Define $\alpha_0=1$, and for $n\geq1$, let
\begin{equation*}\alpha_n=\sum_{p\in P_n}{(-1)}^{l(p)}\prod_{j=1}^{l(p)}\widehat{\mu}(p_j).\end{equation*}
Let $\{g_n\}_{n=0}^{\infty}$ be as defined in $\eqref{gs}$. Then for all $n\in\mathbb{N}_0$,
$$g_n=\sum_{j=0}^{n}\overline{\alpha_{n-j}}e_{j}.$$
\end{lemma}

\begin{proof}
Clearly, $g_0=e_0$ and $g_1=e_1-\langle e_1,e_0\rangle e_0=e_1-\overline{\widehat{\mu}(1)}e_0$. We have that $P_1=\{(1)\}$, so
$$\alpha_1=(-1)^{1}\widehat{\mu}(1)=-\widehat{\mu}(1).$$
So, the conclusion holds for $n=0,1$. Suppose that the conclusion holds up to some $n\in\mathbb{N}$. We have that
\begin{align*}g_{n+1}&=e_{n+1}-\sum_{j=0}^{n}\langle e_{n+1},e_j\rangle g_j\\
&=e_{n+1}-\sum_{j=0}^{n}\overline{\widehat{\mu}(n+1-j)}g_j\\
&=e_{n+1}-\sum_{j=0}^{n}\overline{\widehat{\mu}(n+1-j)}\left(\sum_{k=0}^{j}\overline{\alpha_{j-k}}e_{k}\right)\\
&=e_{n+1}-\sum_{j=0}^{n}\sum_{k=0}^{j}\overline{\widehat{\mu}(n+1-j)}\overline{\alpha_{j-k}}e_{k}\\
&=e_{n+1}-\sum_{k=0}^{n}\sum_{j=k}^{n}\overline{\widehat{\mu}(n+1-j)}\overline{\alpha_{j-k}}e_k.
\end{align*}
Thus, it remains only to show that
$$\alpha_{n+1-k}=-\sum_{j=k}^{n}\widehat{\mu}(n+1-j)\alpha_{j-k}.$$
We have:
\begin{align*}-\sum_{j=k}^{n}\widehat{\mu}(n+1-j)\alpha_{j-k}&=-\sum_{j=k}^{n}\widehat{\mu}(n+1-j)\sum_{p\in P_{j-k}}{(-1)}^{l(p)}\prod_{w=1}^{l(p)}\widehat{\mu}(p_w)\\
&=\sum_{j=k}^{n}\sum_{p\in P_{j-k}}{(-1)}^{l(p)+1}\widehat{\mu}(n+1-j)\prod_{w=1}^{l(p)}\widehat{\mu}(p_w)\\
&=\sum_{j=1}^{n+1-k}\sum_{p\in P_{n-k+1-j}}{(-1)}^{l(p)+1}\widehat{\mu}(j)\prod_{w=1}^{l(p)}\widehat{\mu}(p_w)
\end{align*}
The last equality is obtained by reindexing the sum $j \mapsto n+1-j$.  Now, if $p=(p_1,\ldots,p_{l(p)})\in P_{n}$, then it is obvious that $p_1\in\{1,2,\ldots,n\}$ and $(p_2,p_3,\ldots,p_{l(p)})\in P_{n-p_1}$ (where we define $P_0=\varnothing$). Conversely, if $p_1\in\{1,2,\ldots,n\}$ and $(p_2,p_3,\ldots,p_{l(p)})\in P_{n-p_1}$, then clearly $(p_1,p_2,\ldots,p_{l(p)})\in P_{n}$. Thus, it follows that
$$-\sum_{j=k}^{n}\widehat{\mu}(n+1-j)\alpha_{j-k}=\sum_{p\in P_{n+1-k}}{(-1)}^{l(p)}\prod_{w=1}^{l(p)}\widehat{\mu}(p_w)=\alpha_{n+1-k}.$$
This completes the proof.
\end{proof}

\begin{remark*}Lemma \ref{gnformula} can easily be generalized to any Hilbert space setting in which the $\{g_n\}_{n=0}^{\infty}$ are induced by a stationary sequence $\{\varphi_n\}_{n=0}^{\infty}$ simply by replacing $\widehat{\mu}(m)$ with $a_m$ in all instances, where the $a_m$ are as defined after Definition \ref{stationdef}.\end{remark*}

It should be pointed out that sequence of scalars $\{\alpha_n\}_{n=0}^{\infty}$ depends only on the measure $\mu$. In a general Hilbert space setting where we may not have stationarity, an expansion of the $\{g_n\}$ in terms of the sequence $\{\varphi_n\}$ to which they are associated by $\eqref{gnformula}$ can be described by using inversion of an infinite lower-triangular Gram matrix. For a treatment, see \cite{HalSzw05}.

\begin{definition}Define a Fourier transform of $f$ by

\begin{equation}\mathcal{F}f(y)=\widehat{f}(y):=\int_{0}^{1}f(x)e^{-2\pi iyx}\,d\mu(x).\end{equation}
Observe that
\begin{equation*}\left\lvert\mathcal{F}f(y)\right\rvert=\left\lvert\langle f,e_y\rangle\right\rvert\leq\lVert f\rVert_{L^2(\mu)}\cdot\lVert e_y\rVert_{L^2(\mu)}=\lVert f\rVert_{L^2(\mu)}.\end{equation*}
Thus $\mathcal{F}$ is a linear operator from $L^2(\mu)$ to $L^\infty(\mathbb{R})$ with operator norm $\lVert\mathcal{F}\rVert=1$.\end{definition}

\begin{corollary}\label{maincor} Assume the conditions and definitions of Theorem \ref{mainthm}. Then the coefficients $c_n$ may be expressed
$$c_n=\sum_{j=0}^{n}\alpha_{n-j}\widehat{f}(j),$$
and as a result
$$f(x)=\sum_{n=0}^{\infty}\left(\sum_{j=0}^{n}\alpha_{n-j}\widehat{f}(j)\right)e^{2\pi inx},$$
where the convergence is in norm.
\end{corollary}

\begin{proof}We compute:
\[ c_n = \langle f,g_n\rangle = \left\langle f,\sum_{j=0}^{n}\overline{\alpha_{n-j}}e_j\right\rangle = \sum_{j=0}^{n}\alpha_{n-j}\widehat{f}(j). \]
The second formula then follows by substitution into $\eqref{maineq}$.
\end{proof}

While we have Parseval's identity $\| f \|^2 = \sum_{n=0}^{\infty} | c_{n} |^2 $ as demonstrated by Theorem \ref{mainthm}, in general the lack of the Bessel condition means that $\| f \|^2 \lesssim \sum_{n=0}^{\infty} |\widehat{f}(n)|^2$ does not hold.  In fact, Proposition 3.10 in \cite{DHSW11} demonstrates an example of a function where $\sum_{n=0}^{\infty}|\widehat{f}(n)|^2=+\infty$.

\subsection{Non-Uniqueness of Fourier Coefficients}
We begin with an example. In \cite{JP98}, it was shown that the quaternary Cantor measure $\mu_4$ possesses an orthonormal basis of exponentials. This basis is $\{e_\lambda\}_{\lambda\in\Lambda}$, where the spectrum $\Lambda$ is given by
$$\Lambda=\left\{\sum_{n=0}^{k}\alpha_n4^n:\alpha_n\in\{0,1\},k\in\mathbb{N}_0\right\}=\{0,1,4,5,16,17,20,21,\ldots\}.$$
As a result, any vector $f\in L^2(\mu_4)$ may be written as
$$f=\sum_{\lambda\in\Lambda}\langle f,e_\lambda\rangle e_\lambda,$$
 where the convergence is in the $L^2(\mu_4)$ norm. Notice that if we define a sequence of vectors $\{h_n\}_{n=0}^{\infty}$ by
$$h_n=\begin{cases}e_n&\text{if }n\in\Lambda\\0&\text{otherwise,}\end{cases}$$
we have that
$$\sum_{n=0}^{\infty}\langle f,h_n\rangle e_n=\sum_{\lambda\in\Lambda}\langle f,e_\lambda\rangle e_\lambda=f.$$
On the other hand, since $\mu_4$ is a singular probability measure, by Theorem $\ref{mainthm}$  we also have
$$f=\sum_{n=0}^{\infty}c_ne_n=\sum_{n=0}^{\infty}\langle f,g_n\rangle e_n.$$
It can easily be checked that $h_0=g_0=e_0$ and $h_1=g_1=e_1$, but that $g_{2} \neq h_2=0$.  Thus, the sequences $\{g_{n}\}$ and $\{h_{n} \}$ yield different expansions for general $f \in L^2(\mu_{4})$.

We can again use the Kaczmarz algorithm to generate a large class of sequences $\{h_n\}$ such that $\sum\langle f,h_n\rangle e_n=f$ in the $L^2(\mu)$ norm as follows.  We use $\langle \cdot , \cdot \rangle_{\mu}$ to denote the scalar product in $L^2(\mu)$.

\begin{theorem} \label{T:reproduce}
Let $\mu$ be a singular Borel probability measure on $[0,1)$. Let $\nu$ be another singular Borel probability measure on $[0,1)$ such that $\nu\perp\mu$. Let $0<\eta\leq1$, and define $\lambda:=\eta\mu+(1-\eta)\nu$. Let $\{h_n\}$ be the sequence associated to $\{e_n\}$ in $L^2(\lambda)$ via the Kaczmarz algorithm in Equation \eqref{gs}. Then for all $f\in L^2(\mu)$,
\begin{equation} \label{Eq:reproduce}
f=\sum_{n=0}^{\infty}{\langle f,\eta h_n\rangle}_{\mu} e_n
\end{equation}
in the $L^2(\mu)$ norm. Moreover, if $\lambda^\prime=\eta^\prime\mu+(1-\eta^\prime)\nu^\prime$ also satisfies the hypotheses, then $\lambda^\prime\neq\lambda$ implies $\{\eta^\prime h_n^\prime\}\neq\{\eta h_n\}$ in $L^2(\mu)$. \end{theorem}

\begin{proof}Because $\nu\perp\mu$, there exist disjoint Borel sets $A$ and $B$ such that $A\cup B=[0,1)$, $\mu(B)=0$, and $\nu(A)=0$. Since $\lambda$ is a singular Borel probability measure, the exponentials $\{e_n\}_{n=0}^{\infty}$ are effective in $L^2(\lambda)$. Let $\{h_n\}$ denote the sequence associated to $\{e_n\}$ in $L^2(\lambda)$ via Equation $\eqref{gs}$. Let $f\in L^2(\mu)$, and define $\tilde{f}=f\cdot\chi_{A}$. Clearly, $\tilde{f}\in L^2(\lambda)$.\\

We have that
$$\tilde{f}=\sum_{n=0}^{\infty}{\left\langle \tilde{f},h_n\right\rangle}_{\lambda} e_n$$
in the $L^2(\lambda)$ norm. Now, note that
\begin{align*}\langle f,\eta h_n\rangle_{\mu} &=\int_{0}^{1}f(x)\overline{\eta h_n(x)}\,d\mu(x)\\
&=\int_{A}f(x)\overline{h_n(x)}\,d\lambda\\
&=\langle\tilde{f},h_n\rangle_{\lambda}.
\end{align*}
Therefore,
\[ \lim_{N \to \infty} \left\lVert \tilde{f} - \sum_{n=0}^{N} \langle f, \eta h_{n} \rangle_{\mu} e_{n} \right\rVert^{2}_{L^2(\lambda)} = 0. \]
Since
\[ \left\lVert f-\sum_{n=0}^{N}\langle f,\eta h_n\rangle_{\mu}e_n\right\rVert^2_{L^2(\mu)} \leq \frac{1}{\eta}\left\lVert\tilde{f}-\sum_{n=0}^{N}\langle f,\eta h_n\rangle_{\mu} e_n\right\rVert^2_{L^2(\lambda)}, \]
Equation (\ref{Eq:reproduce}) follows with convergence in $L^2(\mu)$.

It remains only to show that different measures $\lambda$ generate different sequences $\{\eta h_n\}$. Therefore, suppose $\nu^\prime$ is another singular Borel probability measure on $[0,1)$ such that $\nu^\prime\perp\mu$, and let $0<\eta^\prime\leq1$. Set $\lambda^\prime=\eta^\prime\mu+(1-\eta^\prime)\nu^\prime$, and let $\{h_n^\prime\}$ be the sequence associated to $\{e_n\}$ in $L^2(\lambda^\prime)$ via Equation \eqref{gs}. Suppose that $\lambda\neq\lambda^\prime$. We wish to show that $\{\eta h_n\}\neq\{\eta^\prime h^\prime_n\}$ in $L^2(\mu)$.

If $\eta\neq\eta^\prime$, then $\eta h_0=\eta e_0\neq\eta^\prime e_0=\eta^\prime h_0^\prime$ in $L^2(\mu)$. Therefore, assume that $\eta=\eta^\prime$. By virtue of the F. and M. Riesz Theorem, since $\lambda\neq\lambda^\prime$, there must exist an integer $n$ such that $\widehat{\lambda}(n)\neq\widehat{\lambda^\prime}(n)$.  Following \cite{HalSzw05}, we define a lower-triangular Gram matrix $G$ of the nonnegative integral exponentials by
$$(G)_{ij}=\begin{cases}\langle e_i,e_j\rangle=\widehat{\lambda}(j-i)&\text{if }i\geq j\\0&\text{otherwise}\end{cases},$$
and then the inverse of this matrix determines the sequence $\{h_n\}$ associated to $\{e_n\}$ in $L^2(\lambda)$ via
$h_n=\sum_{i=0}^{n}\overline{\alpha_{n-i}}e_i$ where $\alpha_{n-i}=\overline{(G^{-1})_{ni}}$. See \cite{HalSzw05} for details.
($G$ and $G^{-1}$ are stratified since $\{e_n\}$ is stationary.) Therefore, the sequences of scalars $\{\alpha_n\}_{n=0}^{\infty}$ and $\{\alpha^\prime_n\}_{n=0}^{\infty}$ induced by $\lambda$ and $\lambda^\prime$, respectively, in Lemma \ref{gnformula} differ. Let $n$ be the smallest positive integer such that $\alpha_n\neq\alpha_n^\prime$. Then since $\eta=\eta^\prime$, we have
$$\eta^\prime h_n^\prime-\eta h_n=\eta\sum_{j=0}^{n}\left(\overline{\alpha^\prime_{n-j}}-\overline{\alpha_{n-j}}\right)e_j=\eta(\overline{\alpha_{n}-\alpha^\prime_n})e_0 \neq 0.$$
Thus, $\{\eta h_n\}$ and $\{\eta^\prime h_n^\prime\}$ are distinct sequences in $L^2(\mu)$.\end{proof}

\begin{remark}
We note that any convex combination of sequences $\{h_n\}$ that satisfy Equation (\ref{Eq:reproduce}) will again satisfy that equation.
\end{remark}

In general, for a fixed $f \in L^{2}(\mu)$ the set of coefficient sequences $\{ d_{n}\}$ that satisfy $f = \sum_{n=0}^{\infty} d_{n} e_{n}$ can be parametrized by sequences $\{\gamma_n\}$ of scalars satisfying $\sum_{n=0}^{\infty}\gamma_n e_n=0$ via $d_{n} = \langle f, g_{n} \rangle_{\mu} + \alpha_{n}$.  Clearly, Theorem \ref{T:reproduce} is not a complete description of all Fourier series expansions for $f$.

\subsection{Connection to the Normalized Cauchy Transform}
The series $\sum_{n=0}^{\infty}\langle f,g_n\rangle e_n$ given by Theorem \ref{mainthm} is the boundary function of the analytic function $\sum_{n=0}^{\infty}\langle f,g_n\rangle z^n$ on $\mathbb{D}$. This function is in the classical $H^2$ Hardy space since the coefficients are square summable. An intriguing connection between the Kaczmarz algorithm and de Branges-Rovnyak spaces is given by the observations that follow. 

Given a positive Borel measure $\mu$ on $[0,1)$, define a map $V_\mu$, called the normalized Cauchy transform, from $L^1(\mu)$ to the functions defined on $\mathbb{C}\setminus\mathbb{T}$ by
$$V_\mu f(z):=\frac{\int_{0}^{1}\frac{f(e^{2\pi ix})}{1-ze^{-2\pi ix}}\,d\mu(x)}{\int_{0}^{1}\frac{1}{1-ze^{-2\pi ix}}\,d\mu(x)}.$$
Poltoratski\u{i} proved in \cite{Pol93} that $V_{\mu}$ maps $L^2(\mu)$ to the de Branges-Rovnyak space $\mathcal{H}(b)$, where $b(z)$ is the inner function associated to $\mu$ via the Herglotz representation theorem.  Poltoratski\u{i} also proved that $V_{\mu}$ is the inverse of a unitary operator that is a rank one perturbation of the unilateral shift as given by Clark \cite{Clark72}, and hence $V_{\mu}$ is unitary.

\begin{proposition}  \label{P:mainprop}
Assume the hypotheses of Theorem \ref{mainthm}. Then for $z\in\mathbb{D}$,
$$V_\mu f(z) = \sum_{n=0}^{\infty}\langle f,g_n\rangle z^n.$$
\end{proposition}

\begin{proof}Define
\begin{equation}\label{Fdef}F(z):=\int_{0}^{1}\frac{1}{1-ze^{-2\pi ix}}\,d\mu(x).\end{equation}
That is, $F(z)$ is the Cauchy integral of $\mu$, which is analytic on $\mathbb{D}$. It is easily seen that
\begin{equation*}F(z)=\sum_{n=0}^{\infty}\widehat{\mu}(n)z^n.\end{equation*}
By \eqref{Fdef}, $\text{Re}(F(z))>1/2$ for $z\in\mathbb{D}$, and hence, $1/F(z)$ is also analytic on $\mathbb{D}$. Writing $1/F(z)=\sum_{n=0}^{\infty}c_nz^n$, we have $1=\sum_{n=0}^{\infty}\left(\sum_{k=0}^{n}c_k\widehat{\mu}(n-k)\right)z^n$, and so $\sum_{k=0}^{n}c_k\widehat{\mu}(n-k)=0$ for all $n\geq1$. Then using \eqref{gs}, an inductive argument shows that $g_n=\sum_{i=0}^{n}\overline{c_{n-i}}e_i$ for all $n$. The $c_n$ are unique by Gaussian elimination, so in fact $c_n=\alpha_n$ for all $n$, the $\alpha_n$ as in Lemma $\ref{gnformula}$. Hence,
$$\frac{1}{F(z)}=\sum_{n=0}^{\infty}\alpha_nz^n.$$
It is also clear that
$$\int_{0}^{1}\frac{f(e^{2\pi ix})}{1-ze^{-2\pi ix}}\,d\mu(x)=\sum_{n=0}^{\infty}\langle f,e_n\rangle z^n.$$
Therefore, we have
\begin{align*}\frac{\int_{0}^{1}\frac{f(e^{2\pi ix})}{1-ze^{-2\pi ix}}\,d\mu(x)}{\int_{0}^{1}\frac{1}{1-ze^{-2\pi ix}}\,d\mu(x)}&=\left(\sum_{n=0}^{\infty}\langle f,e_n\rangle z^n\right)\left(\sum_{m=0}^{\infty}\alpha_mz^m\right)\\
&=\sum_{n=0}^{\infty}\left(\sum_{i=0}^{n}\langle f,\overline{\alpha_{n-i}}e_i\rangle\right)z^n\\
&=\sum_{n=0}^{\infty}\langle f,g_n\rangle z^n.
\end{align*}
\end{proof}

Two of the main results in \cite{Pol93} are Theorems 2.5 and 2.7, which together show that the Fourier series of $V_{\mu}f(z)$ converges to $f$ in the $L^2(\mu)$ norm provided that $\mu$ is singular.  Combining this together with Proposition \ref{P:mainprop} recovers our Theorem \ref{mainthm}.  Adding Clark's result that implies that $V_{\mu}$ is unitary, and we recover the Plancherel identity.

Poltoratski\u{i}'s results are more general than our Theorem \ref{mainthm} in the following way:  if $\mu$  has an absolutely continuous component and a singular component, then for any $f \in L^2(\mu)$, the Fourier series of $V_{\mu} f$ converges to $f$ in norm with respect to the singular component.  The Fourier series cannot in general converge to $f$ with respect to the absolutely continuous component of $\mu$ since the nonnegative exponentials are incomplete.  It is unclear whether for such a $\mu$ every $f$ can be expressed in terms of a bi-infinite Fourier series.  For singular $\mu$, our Theorem \ref{mainthm} guarantees norm convergence of the Fourier series of $V_{\mu} f$ to $f$ as do Poltoratski\u{i}'s results.  However, Poltoratski\u{i} also comments in \cite{Pol93} that the Fourier series converges pointwise $\mu$-a.e. to $f$.


\section{A Shannon Sampling Formula}

In \cite{Str00}, Strichartz introduces a sampling formula for functions that are bandlimited in a generalized sense.  He considers functions whose spectra are contained in a certain compact set $K$ that is the support of a spectral measure $\mu$.  If $F$ is a strongly $K$-bandlimited function, then he shows that it has an expression
\begin{equation*}F(x)=\sum_{\lambda\in\Lambda}F(\lambda)\widehat{\mu}(x-\lambda),
\end{equation*} where $\Lambda$ is a spectrum for $L^2(\mu)$. 

We will now prove a similar sampling formula for analogously bandlimited functions.  Our formula does not rely on an exponential basis and hence holds even for non-spectral singular measures. (Indeed, it even holds for singular measures devoid of exponential frames.) The price paid for not using an exponential sequence dual to itself is that the samples $F(\lambda)$ are replaced by the less tidy $\sum_{j=0}^{n}\alpha_{n-j} F(j)$.   

\begin{theorem}Let $\mu$ be a singular Borel probability measure on $[0,1)$. Let $\{\alpha_i\}_{i=0}^{\infty}$ be the sequence of scalars induced by $\mu$ by Lemma $\ref{gnformula}$. Suppose $F:\mathbb{R}\rightarrow\mathbb{C}$ is of the form
$$F(y)=\int_{0}^{1}f(x)e^{-2\pi iyx}\,d\mu(x)$$ for some $f\in L^2(\mu)$. Then
\begin{equation*}
F(y)=\sum_{n=0}^{\infty}\left(\sum_{j=0}^{n}\alpha_{n-j}F(j)\right)\widehat{\mu}(y-n),
\end{equation*}
where the series converges uniformly in $y$. 
\end{theorem}

\begin{proof}
By Theorem \ref{mainthm}, $f$ may be expressed $f=\sum_{n=0}^{\infty}c_ne_n$, the convergence occurring in the $L^2(\mu)$ norm. We compute:
\begin{align*}F(y)&=\int_{0}^{1}f(x)e^{-2\pi iyx}\,d\mu(x)\\
&=\langle f,e_y\rangle\\
&=\left\langle\sum_{n=0}^{\infty}c_n e_n,e_y\right\rangle\\
&=\sum_{n=0}^{\infty}c_n\langle e_n,e_y\rangle\\
&=\sum_{n=0}^{\infty}c_n\widehat{\mu}(y-n).
\end{align*}

Recall from Corollary \ref{maincor} that
$$c_n=\sum_{j=0}^{n}\alpha_{n-j}\widehat{f}(j)=\sum_{j=0}^{n}\alpha_{n-j}F(j),$$
where the $\alpha_n$ are defined by Lemma \ref{gnformula}. Combining these computations, we obtain that for any $y\in\mathbb{R}$,
\begin{equation}\label{thm2maineq}F(y)=\sum_{n=0}^{\infty}\left(\sum_{j=0}^{n}\alpha_{n-j}F(j)\right)\widehat{\mu}(y-n).\end{equation}
Let $S_k:=\sum_{n=0}^{k}c_n e_n$. Since $S_k\rightarrow f$ in the $L^2(\mu)$ norm and the Fourier transform $\mathcal{F}:L^2(\mu)\rightarrow L^\infty(\mathbb{R})$ is bounded, $\{\mathcal{F}S_k\}\rightarrow\mathcal{F}f$ in $L^\infty(\mathbb{R})$. Then because $\mathcal{F}S_k(y)=\sum_{n=0}^{k}c_n\widehat{\mu}(y-n)$, we have that $\sum_{n=0}^{\infty}c_n\widehat{\mu}(y-n)$ and hence $\eqref{thm2maineq}$ converge uniformly in $y$ to $\mathcal{F}f(y)$.
\end{proof}

It should be noted that, in contradistinction to the sampling formula of Strichartz, the convergence of the series in Equation (\ref{thm2maineq}) does not follow from the Cauchy-Schwarz inequality, because it is possible that $\sum_{n=0}^{\infty} | \widehat{\mu}(y-n) |^2 = + \infty$.





\bibliographystyle{amsalpha}
\bibliography{fssm}

\providecommand{\bysame}{\leavevmode\hbox to3em{\hrulefill}\thinspace}
\providecommand{\MR}{\relax\ifhmode\unskip\space\fi MR }
\providecommand{\MRhref}[2]{%
  \href{http://www.ams.org/mathscinet-getitem?mr=#1}{#2}
}
\providecommand{\href}[2]{#2}
\begin{thebibliography}{DHSW11}

\bibitem[Cla72]{Clark72}
Douglas~N. Clark, \emph{One dimensional perturbations of restricted shifts}, J.
  Analyse Math. \textbf{25} (1972), 169--191. \MR{0301534 (46 \#692)}

\bibitem[DHSW11]{DHSW11}
Dorin~Ervin Dutkay, Deguang Han, Qiyu Sun, and Eric Weber, \emph{On the
  {B}eurling dimension of exponential frames}, Adv. Math. \textbf{226} (2011),
  no.~1, 285--297. \MR{2735759 (2012a:42058)}

\bibitem[DHW14]{DHW14}
Dorin~Ervin Dutkay, Deguang Han, and Eric Weber, \emph{Continuous and discrete
  {F}ourier frames for fractal measures}, Trans. Amer. Math. Soc. \textbf{366}
  (2014), no.~3, 1213--1235. \MR{3145729}

\bibitem[DS52]{DS52}
R.~J. Duffin and A.~C. Schaeffer, \emph{A class of nonharmonic {F}ourier
  series}, Trans. Amer. Math. Soc. \textbf{72} (1952), 341--366. \MR{0047179
  (13,839a)}

\bibitem[HS05]{HalSzw05}
Rainis Haller and Ryszard Szwarc, \emph{Kaczmarz algorithm in {H}ilbert space},
  Studia Math. \textbf{169} (2005), no.~2, 123--132. \MR{2140451 (2006b:41049)}

\bibitem[JP98]{JP98}
Palle E.~T. Jorgensen and Steen Pedersen, \emph{Dense analytic subspaces in
  fractal {$L^2$}-spaces}, J. Anal. Math. \textbf{75} (1998), 185--228.
  \MR{1655831 (2000a:46045)}

\bibitem[Kac37]{Kacz37}
Stefan Kaczmarz, \emph{Angen\"{a}herte aufl\"{o}sung von systemen linearer
  gleichungen}, Bulletin International de l'Acad\'{e}mie Plonaise des Sciences
  et des Lettres. Classe des Sciences Math\'{e}matiques et Naturelles.
  S\'{e}rie A. Sciences Math\'{e}matiques \textbf{35} (1937), 355--357.

\bibitem[KM01]{KwMy01}
Stanis{\l}aw Kwapie{\'n} and Jan Mycielski, \emph{On the {K}aczmarz algorithm
  of approximation in infinite-dimensional spaces}, Studia Math. \textbf{148}
  (2001), no.~1, 75--86. \MR{1881441 (2003a:60102)}

\bibitem[Lai12]{Lai12}
Chun-Kit Lai, \emph{Spectral analysis on fractal measures and tiles}, Ph.d.
  dissertation, The Chinese University of Hong Kong, 2012, Published by
  ProQuest, LLC: UMI Dissertation Publishing, Number 3535691 (2013).

\bibitem[LO01]{LiOg01}
Shidong Li and Hidemitsu Ogawa, \emph{Pseudo-duals of frames with
  applications}, Appl. Comput. Harmon. Anal. \textbf{11} (2001), no.~2,
  289--304. \MR{1848709 (2002f:46014)}

\bibitem[Pol93]{Pol93}
A.~G. Poltoratski{\u\i}, \emph{Boundary behavior of pseudocontinuable
  functions}, Algebra i Analiz \textbf{5} (1993), no.~2, 189--210, English
  translation in St. Petersburg Math. 5:2 (1994): 389--406. \MR{1223178
  (94k:30090)}

\bibitem[RR16]{Riesz16}
Frigyes Riesz and Marcel Riesz, \emph{\"{U}ber die randwerte einer analytischen
  funktion}, Quatri\`{e}me Congr\`{e}s des Math\'{e}maticiens Scandinaves
  (1916), 27--44, See Rudin, \textit{Real and Complex Analysis, 3rd Edition}:
  pg. 341.

\bibitem[Str00]{Str00}
Robert~S. Strichartz, \emph{Mock {F}ourier series and transforms associated
  with certain {C}antor measures}, J. Anal. Math. \textbf{81} (2000), 209--238.
  \MR{1785282 (2001i:42009)}

\end{thebibliography}

\end{document}